\documentclass[a4paper,12pt]{amsart}
\usepackage{amsmath}
\usepackage{amsthm}
\usepackage[hyphens]{url}
\usepackage{hyperref}
\usepackage[mathscr]{euscript}
\usepackage{xcolor}

\newtheorem{thm}{Theorem}
\newtheorem{lem}{Lemma}

\newtheorem{rmk}{Remark}
\newtheorem*{prop*}{Proposition}
\def\Q{\mathbb{Q}}
\def\R{\mathbb{R}}
\def\N{\mathbb{N}}
\def\M{\mathcal{M}}

\def\d{\mathrm{d}}

\def\half{\frac{1}{2}}

\title[Omega Results for The Divisor and Circle Problems]{Omega Results for The Divisor and Circle Problems Using The Resonance Method}

\begin{document}
\author{Kamalakshya Mahatab}
	\address{Kamalakshya Mahatab\\ Department of Mathematics \\
		Indian Institute of Technology Kharagpur \\
			Kharagpur-721302,  India.} 
	\email{kamalakshya@maths.iitkgp.ac.in}

\maketitle
\begin{abstract}
We apply the resonance method to obtain large values of general exponential sums with positive coefficients. As applications, we show improved $\Omega$-bounds for Dirichlet and Piltz divisor problems, Gauss circle Problem, and error terms for the mean square of the Riemann zeta function and the Dirichlet $L$-functions. 
\end{abstract}

\section{Introduction}
Let $d(n)$ be the number of divisors of $n$. Then 
\[\sum_{n\leq x} d(n)= x\log x + (2\gamma-1)x+\Delta(x), \]
where $\Delta(x)$ is the error term. The famous Dirichlet divisor problem asks for the best possible upper bound for $\Delta(x)$, which remains unsolved. It is widely conjectured that
\[\Delta(x)\ll x^{1/4+\epsilon} \text{ for any } \epsilon>0.\]
However, the best known upper bound for $\Delta(x)$ is due to Huxley \cite{hux}, who proved 
\[\Delta(x)=O\left(x^{\frac{131}{416}}(\log x)^{\frac{26947}{8320}}\right).\]
We investigate the above problem from a different perspective. We may ask, how large $\Delta(x)$ could be when $x\rightarrow\infty$. In this direction, the best known $\Omega~\text{bound}$ for $\Delta(x)$ is due to Soundararajan\cite{So}\footnote{Throughout this paper we will use the notations $\log_2 X := \log \log X \text{ and } \log_3 X = \log \log \log X.$}:
\begin{align*}
 \Delta(x)=\Omega\left((x\log x)^{1/4}(\log_2 x)^{3/4(2^{4/3}-1)}(\log_3 x)^{-5/8}\right). 
\end{align*}
He conjectures that the power of $\log_2 x$ is optimal. So we only hope to improve on the power of $\log_3 x$. In this paper, we improve the power from $-5/8$ to $-3/8$.
\begin{thm}\label{thm:divisor}
We have
 \[\max_{X/2<x\leq 5X^{3/2}(\log X)^2}\left|\frac{\Delta(x^2)}{\sqrt x}\right|
  \gg (\log X)^{1/4} (\log_2X)^{ (3/4)(2^{4/3}-1)}(\log_3 X)^{-3/8}.\]
\end{thm}
In particular,
\[\Delta(x)=\Omega\left((x\log x)^{1/4}(\log_2 x)^{3/4(2^{4/3}-1)}(\log_3 x)^{-3/8}\right).\]

We may explore the possibility of removing the factor $(\log_3 x)^{-3/8}$ by trying to demonstrate that there are at least $\alpha(\log \alpha)^{\lambda-1-\lambda\log\lambda}$ integers in $[\alpha, C\alpha]$ that have more than $\lambda \log_2 \alpha$ distinct prime factors. However, the choice of $\alpha$ is limited in order to keep the error under control, which makes this approach challenging.

By the Vorono\"i summation formula\cite{tit}, $\Delta(x^2)$ can be represented by the conditional covergent series
\[\Delta(x^2)=\frac{x^{1/2}}{\pi\sqrt{2}}\sum_{n=1}^{\infty}\frac{d(n)}{n^{3/4}}\cos\left(4\pi\sqrt n x -\frac{\pi}{4}\right).\]
To improve workability, we may truncate the sum mentioned above while ensuring a manageable error margin \cite{So}:
\[\Delta(x^2)=\frac{x^{1/2}}{\pi\sqrt{2}}\sum_{n\leq X^{3}}\frac{d(n)}{n^{3/4}}\cos\left(4\pi\sqrt n x -\frac{\pi}{4}\right) + O(X^\epsilon),\]
uniformly for $\sqrt X \leq x\leq X^{3/2}$, and for any $\epsilon>0$. 

We will use the resonance method on the exponential sum 
\[\sum_{n\leq X^{3}}\frac{d(n)}{n^{3/4}}\cos\left(4\pi\sqrt n x -\frac{\pi}{4}\right),\]
to get large values. The choice of our resonator is inspired by the work of Aistleitner, Mahatab, and Munsch\cite{aist2}. Previously, the resonance method has been widely used to identify large values of $L-$functions\cite{aist, BS1, dlb, So}. This paper aims to extend the application of the resonance method to general exponential sums defined as follows:
\[F_\beta(x)=\sum_{n\leq X^{A_1}}a_n \cos(x\lambda_n+\beta),\]
where $\lambda_n $ and $a_n$ are assumed to be positive. A detailed lowerbound for this sum is given in Theorem~\ref{thm:main}. Although the lower bound obtained here is explicit, there is potential for improvement on the explicit constant.

As our method for showing large values of $\Delta(x)$ is general in nature, it can be applied to the Circle problem, the Piltz divisor problem, and the error term for the mean square of the Riemann zeta functionand Dirichlet $L$-functions. Below, we list these applications, while the details of the proofs are discussed in  Section~\ref{sec:applications}.

\begin{thm}[Circle Problem] ~\label{thm:circle}

Let 
\[P(x)=\sum_{n\leq x}r(n)-\pi x,\]
where $r(n)$ denotes the number of ways of writing $n$ as sum of two squares. Then
 \[\max_{X/2<x\leq 5X^{3/2}(\log X)^2}\left|\frac{P(x^2)}{\sqrt x}\right|
  \gg (\log X)^{1/4} (\log_2X)^{ (3/4)(2^{4/3}-1)}(\log_3 X)^{-3/8}.\]  
\end{thm}

\begin{thm}[Piltz Divisor Problem] ~\label{thm:piltz} 

Let 
\[\Delta_k(x)=\sum_{n\leq x}d_k(n)-\text{Res}_{s=1}\zeta(s)^k \frac{x^s}{s},\]
where $d_k(n)$ denotes the number of ways of expressing $n$ as product of $k$ factors. Then
 \[\max_{X/2<x\leq 5X^{3/2}(\log X)^2} \left|\frac{\Delta_k(x^k)}{x^{\frac{k-1}{2}}}\right|
  \gg (\log X)^{(k-1)/2k} (\log_2X)^{ ((k+1)/2k)(k^{2k/(k+1)}-1)}(\log_3 X)^{-1/2+\frac{k-1}{4k}}.\]  
  We may replace $|\Delta_k(x)|$ with $\Delta_k(x)$ when $k\equiv 3 \pmod 8$ and with $-\Delta_k(x)$ when $k\equiv 7\pmod 8$.
\end{thm}

Theorem~\ref{thm:circle} and \ref{thm:piltz} improves the following bounds of Soundararajan\cite{So}
\begin{align*}
 P(x)&=\Omega\left((x\log x)^{1/4} (\log_2 x)^{ (3/4)(2^{4/3}-1)}(\log_3 x)^{-5/8}\right),\\
 \Delta_k(x)&=\Omega\left((x\log x)^{(k-1)/2k} (\log_2 x)^{ ((k+1)/2k)(k^{2k/(k+1)}-1)}(\log_3 x)^{-1/2-\frac{k-1}{4k}}\right),
\end{align*}
where the $\Omega$ bound of $\Delta_k(x)$ can be replaced with $\Omega_+$ when $k\equiv 3 \pmod 8$ and with $\Omega_-$ when $k\equiv 7\pmod 8$. The next theorem improves the result of Lau and Tsang \cite{Tsang} by a factor of $(\log_3 X)^{1/4}$.

\begin{thm}[Mean Square of The Riemann Zeta] ~\label{thm:riemann}

Let 
\[E(x)=\int_{0}^{x}\left|\zeta\left(\half + it\right)\right|^2\d t-x\log\frac{x}{2\pi}-(2\gamma-1)x,\]
where $\zeta$ denotes the Riemann zeta function. Then
 \[\max_{X/2<x\leq 5X^{3/2}(\log X)^4}\left|\frac{E(2\pi x^2)}{\sqrt x}\right|
  \gg (\log X)^{1/4} (\log_2 X)^{ (3/4)(2^{4/3}-1)}(\log_3 X)^{-3/8}.\]  
\end{thm}

Similar to Theorem~\ref{thm:riemann}, we can also obtain a lower bound for the error term of mean square formula for Dirichlet $L$-function, which improves the result of Lau and Tsang \cite{Tsang2}.

\begin{thm}[Mean Square of The Dirichlet $L$-function] ~\label{thm:dirichlet}
Let
\[E(q, x)=\sum_{\chi \pmod q}\int_{0}^{x}\left|L\left(\half + it, \chi\right)\right|^2\d t-\frac{\phi(q)^2}{q}x\left(\log\frac{qx}{2\pi}+\sum_{p|q}\frac{\log p}{p-1}+ 2\gamma-1\right).\]
Then
\[\max_{X/2<x\leq 5X^{3/2}(\log X)^4}\left|\frac{E(q, 2\pi x^2)}{\sqrt x}\right|
  \gg (\log X)^{1/4} (\log_2 X)^{ (3/4)(2^{4/3}-1)}(\log_3 X)^{-3/8}.\]  
\end{thm}

We may achieve comparable results for the divisor problem within number fields\cite{GirNowak}\cite{nilmoni2}, the lattice counting problem associated with an n-dimensional sphere \cite{nowak}\cite{KuhNowak}\cite{nilmoni}, the counting problem for Cygan-Koranyi balls \cite{gath}.

In the following section, we will explore the method involving a generic exponential sum.
\section*{Acknowledgment}
\small{We would like to thank Prof. Eero Saksman for his valuable discussions during the preparation of this article. We also appreciate the organizers of the \href{https://nitc.ac.in/conferences-seminars-conferences/international-conference-on-lie-algebra-and-number-theory-iclant---2024---june-10th-14th-2024}{ICLANT-2024}, where the initial version of this paper was presented, which allowed us to incorporate several helpful suggestions from the audience. In particular, we are grateful to Kohji Matsumoto for recommending us to include the application to the mean square of the Riemann zeta function.


The author is supported by the DST INSPIRE Faculty Award Program and grant no. DST/INSPIRE/04/2019/002586.}
\section{Large Values of Exponential Sums}
In this section we will work with general exponential sums having positive coefficients.
Let 
\[F_\beta(x)=\sum_{n\leq X^{A_1}}a_n \cos(x\lambda_n+\beta),\]
where $\lambda_n>0$. 
The same method also works when all $\lambda_n<0$. We will also assume $a_n\geq0$ for all $n\geq 1$.
We will need the following convolution formula for $F_\beta(x)$, which is inspired from \cite{mont}.
\subsection{Convolution Formula}
\begin{lem}\label{lem:convolution_formula}
 Let $\alpha>0$ and $x\in\R$.
 Then 
 \begin{align*}
 &\int_{-\infty}^{\infty}F_\beta(x+u)\left(\frac{\sin \alpha u}{u}\right)^2e^{-i2\alpha u}\d u\\
 &=\half e^{i\beta}\sum_{\lambda_n}a_n w_{\alpha}(\lambda_n)e^{i\lambda_n x}, 
 \end{align*}
where 
\[w_{\alpha}(\lambda_n)=\frac{\pi}{2}\max\{0, \, 2\alpha-|\lambda_n-2\alpha|\}.\]
\end{lem}
\begin{proof}
Let
\[F(x)=\sum_{n\leq X^{A_1}}a_n \exp(ix\lambda_n).\]
Then
\begin{align*}
&\quad\int_{-\infty}^{\infty}F_\beta(x+u)\left(\frac{\sin \alpha u}{u}\right)^2e^{-i2\alpha u}\d u\\
&=\frac{1}{2}\int_{-\infty}^{\infty}\left(F(x+u)e^{i\beta}+F(-x-u)e^{-i\beta}\right)\left(\frac{\sin \alpha u}{u}\right)^2e^{-i2\alpha u}\d u\\
&=\frac{1}{2}e^{i\beta}\sum_{\lambda_n}a_n w_{\alpha}(\lambda_n)e^{i\lambda_n x}+ \frac{1}{2}e^{-i\beta}\sum_{\lambda_n}a_n w_{\alpha}(-\lambda_n)e^{-i\lambda_n x}.
\end{align*}
As $w_{\alpha}(-\lambda_n)=0$ for all $n\geq 1$,
\[\int_{-\infty}^{\infty}F_\beta(x+u)\left(\frac{\sin \alpha u}{u}\right)^2e^{-i2\alpha u}\d u=\frac{1}{2}e^{i\beta}\sum_{\lambda_n}a_n w_{\alpha}(\lambda_n)e^{i\lambda_n x}.\]
\end{proof}

We may note that the above convolution formula makes the phase $\beta$ redundant in computation of large values. Further, it also gives the flexibility to control the length of the exponential sum by choosing $\alpha$ such that the sum runs upto $\lambda_n\leq 2\alpha$. 
\subsection{Construction of Resonator}
The convolution of $F_\beta$ in Lemma~\ref{lem:convolution_formula} concentrates the values of $F_\beta$ in a small neighbourhood of $x$. Next, we need to choose suitable $\lambda_n$s to obtain large values in a neighbourhood of $x$. We will do this using the resonance method.

Let $0<A_4<A_3<A_2<A_1$. Further, let $Y_1= X^{A_3}$ and $Y_2=X^{A_2}$.
To find extreme values of $F_\beta(x)$ for $Y_1/2\leq x\leq 2Y_2(\log Y_2)^2$, 
we define
\begin{align*}
I_1&=\int_{-\infty}^{\infty}\int_ {-X^{A_4}}^{X^{A_4}}F_0(x+u)\left(\frac{\sin \alpha u}{u}\right)^2e^{-i2\alpha u}|R(x)|^2\Phi(x/Y_2)\d u\d x, \text{ and }\\
I_2&= \int_{-\infty}^{\infty} |R(x)|^2\Phi(x/Y_2)\d x,
\end{align*}
where $\Phi(x):=e^{-x^2/2}$. In the resonance method, we use the idea that
\[\frac{|I_1|}{|I_2|}\leq\max_{Y_1/2\leq x\leq 2Y_2\log^2Y_2}|F_0(x)|+ \text{Error}.\]
The transition from $0$ to $\beta$ can be done using Lemma~\ref{lem:convolution_formula} and Lemma~\ref{lem:Fbeta_to_F0}. Infact, the large values we will obtain, will be independent of $\beta$.


To construct the resonator $R(x)$, we will collect the frequencies $e^{i\lambda_n x}$. Let 
\[\M\subseteq \{\lambda_n: C_1\alpha\leq\lambda_n\leq 2\alpha\},\]
for some $0<C_1<2$, such that $\M$ is a linearly independent set over the rationals $\Q$ and let $|\M|=M$.
We define a resonator $R(x)$ on $\M$ as 
\[R(x)=\sum_{u \in \mathbb{N}[\mathcal{M}]} r(u)\exp(iu x),\]
where\footnote{ It is useful to observe that we could define $0<r(\lambda_n)<1$ arbitrarily for $\lambda_n\in \M$, and then extend  to the rest of $\N[\M]$ by just demanding that $r(u+v)=r(u)r(v)$ for all $u, v \in \N[\M]$. This property of our resonator may help to extend it to combinatorial and algebraic setting.} $r(u):=e^{-u/2\alpha}$. Above $\mathbb{N}[\mathcal{M}]$ stands for the set of non-negative integral linear combinations of elements in $\mathcal{M}$.

Due to Linear independence of $\M$ and as $r(u+v)=r(u)r(v)$ for all $u, v \in \M$, we may write $R(x)$ as an Euler product
\begin{equation}\label{eq:euler}
R(x)=\prod_{\lambda_n\in\M}\left(1-r(\lambda_n)\exp(i\lambda_n x)\right)^{-1}.
\end{equation}

Now we are ready to state our main theorem.

\begin{thm}\label{thm:main}
We use the notations as defined before. Then 
\begin{align*}
 &\max_{X^{A_3}/2<x\leq 2A_2^2X^{A_2}(\log X)^2}|F_\beta(x)|\\
 &\geq \frac{\pi}{4e}\sum_{\lambda_n\in \M}a_n+ O\left(X^{A_3-A_2}e^{2M/C_1}\left(\sum_{\lambda_n\leq 4\alpha}a_n\right) + \frac{X^{-A_4}}{\alpha}\left(\sum_{n\leq X^{A_1}}a_n\right)\right).
\end{align*}
\end{thm}

\begin{rmk}
In the above theorem, the constant $\frac{\pi}{4e}$ is not optimal, and its optimization is left for future research.
For example, it can be improved to $\pi/(4e^{C_2/2})$ by choosing $\M\subseteq [C_1\alpha, C_2 \alpha]$.
\end{rmk}

\section{Proof of Theorem~\ref{thm:main}}

We start by proving the following two bounds concerning  $|R(x)|^2$.

\begin{lem}\label{lem:simplify0} 
Consider $I_1$ and $I_2$ as defined above. \\
\noindent  {\rm (i)}\quad For any $x\in \R$  we have 
$$|R(x)|^2\leq \exp\left(\frac{2}{C_1}M\right),$$ and

\smallskip

\noindent {\rm (ii)}\quad $$I_2\geq \sqrt{2\pi}Y_2\exp\left(\frac{M}{7}\right).$$
\end{lem}
\begin{proof} Towards (i) we get
\begin{align*}
|R(x)|^2&
\leq\prod_{\lambda_n\in\M }\left(1-r(\lambda_n)\right)^{-2}\\
&\leq\exp\left(-2\sum_{\lambda_n\in \M}\log(1-e^{-\lambda_n/2\alpha})\right)\\
&\leq\exp\left(-2\sum_{\lambda_n\in \M}\log(1-e^{-C_1/2})\right)\leq\exp\left(\frac{2}{C_1}M\right). 
\end{align*}

Note
\[\exp\left(-2M\log(1-e^{-C_1/2})\right)\leq \exp\left(\frac{2M}{e^{C_1/2}}\right)\leq\exp\left(\frac{2M}{1+C_1/2}\right)\leq \exp\left(\frac{2M}{C_1}\right). \]

To show (ii), we expand the integral for $I_2$ as follows

\begin{align*}
 I_2=\int^\infty_{-\infty}\sum_{u\in \N[\M]}\sum_{v\in \N[\M]}r(u)r(v)e^{i(u-v) x}\Phi(x/Y_2)\d x.\\
 \end{align*}

We drop the terms where $u\neq v$, as each term in the above sum is positive. So we get

\begin{align*}
 I_2&\geq \sum_{u\in \N[\M]} r(u)^2 \int_{-\infty}^{\infty}\Phi(x/Y_2)\d x\\
 &\geq \sqrt{2\pi}Y_2\sum_{u\in \N[\M]} r(u)^2\\
 & \geq \sqrt{2\pi}Y_2 \prod_{{\lambda}_n}(1-e^{-{\lambda}_n/\alpha})^{-1}\geq \sqrt{2\pi}Y_2 \prod_{{\lambda}_n}(1-e^{-2})^{-1}
 \geq \sqrt{2\pi}Y_2\exp\left(\frac{M}{7}\right).
\end{align*}

\end{proof}

From now on, we will work with $F_0(x)$ and relate it to $F_\beta(x)$ in the proof of Theorem~\ref{thm:main} later, using Lemma~\ref{lem:convolution_formula} and Lemma~\ref{lem:Fbeta_to_F0}.

In the following lemma, we use the resonator to obtain a lower bound for $\frac{I_1}{I_2}$. Later in Lemma~\ref{lem:simplifyI1}, we simplify $I_1$ by bounding its tail part.

\begin{lem}\label{lem:I1byI2}
Let $I_1$ and $I_2$ be as before.  Then 
\[\left|\frac{I_1}{I_2}\right| \geq \half \sum_{\lambda_n\in\M}a_n r(\lambda_n)w_{\alpha}(\lambda_n) +O\left(X^{-A_4}\sum_{n\leq X^{A_1}} a_n\right).\] 
\end{lem}
\begin{proof}
Consider
\begin{align} \label{eq:tail_conv_formula}
\notag
 &\left|\int_{-\infty}^{\infty}\int_ {X^{A_4}}^{\infty}F_0(x+u)\left(\frac{\sin \alpha u}{u}\right)^2e^{-i2\alpha u}|R(x)|^2\Phi(x/Y_2)\d u\d x\right|\\
 \notag
 &\ll \left(\sum_{n\leq X^{A_1}} a_n\int_ {X^{A_4}}^{\infty}\left|\frac{\sin \alpha u}{u}\right|^2\d u\right)\int_{-\infty}^{\infty}|R(x)|^2\Phi(x/Y_2)\d x\\
 &\ll \left(X^{-A_4}\sum_{n\leq X^{A_1}} a_n\right) I_2.
\end{align}

Similarly 
\begin{align*}
 &\left|\int_{-\infty}^{\infty}\int^ {-X^{A_4}}_{-\infty}F_0(x+u)\left(\frac{\sin \alpha u}{u}\right)^2e^{-i2\alpha u}|R(x)|^2\Phi(x/Y_2)\d u\d x\right|\\
 &\ll \left(X^{-A_4}\sum_{n\leq X^{A_1}} a_n\right) I_2.
\end{align*}

Using the above inequalities and Lemma~\ref{lem:convolution_formula}, we obtain 
\begin{align*}
 |I_1|
 &=\left|\int_{-\infty}^{\infty}\int_ {-\infty}^{\infty}F_0(x+u)\left(\frac{\sin \alpha u}{u}\right)^2e^{-i2\alpha u}|R(x)|^2\Phi(x/Y_2)\d u\d x \right|\\
 &+O\left(I_2X^{-A_4}\sum_{n\leq X^{A_1}} a_n\right)\\
 &= \left|\half \sum_{\lambda_n}a_n w_{\alpha}(\lambda_n)\int_{-\infty}^{\infty}|R(x)|^2\Phi(x/Y_2)e^{i\lambda_n x}\d x \right|+O\left(I_2X^{-A_4}\sum_{n\leq X^{A_1}} a_n\right).\\
 \end{align*}
 
 To achieve our desired result, we will simplify the sum provided above and incorporate $I_2$ within it. Starting from the second line, we can omit the absolute value, since each term under the integral is positive. This technique helps us to drop as many terms as we need to simplify our calculation. This approach is inspired by \cite{aist2}, where a similar method is employed to derive large values of $\zeta(1+it)$.
 
 \begin{align*}
 &\left|\half \sum_{\lambda_n}a_n w_{\alpha}(\lambda_n)\int_{-\infty}^{\infty}|R(x)|^2\Phi(x/Y_2)e^{i\lambda_n x}\d x \right|\\
 &= \half \sum_{\lambda_n}a_n w_{\alpha}(\lambda_n)\sum_{u\in \N[\M]}\sum_{v\in \N[\M]}r(u)r(v)\int_{-\infty}^{\infty}e^{i(\lambda_n+u-v) x}\Phi(x/Y_2)\d x\\
 &\geq\half \sum_{\lambda_n\in \M}a_n w_{\alpha}(\lambda_n)\sum_{u\in \N[\M]}\sum_{v+\lambda_n\in \N[\M]}r(u)r(v+\lambda_n)\int_{-\infty}^{\infty}e^{i(u-v) x}\Phi(x/Y_2)\d x\\
 & = \half \sum_{\lambda_n\in \M}a_n w_{\alpha}(\lambda_n)r(\lambda_n)\sum_{u\in \N[\M]}\sum_{v\in \N[\M]}r(u)r(v)\int_{-\infty}^{\infty}e^{i(u-v) x}\Phi(x/Y_2)\d x\\
 & = \left(\half \sum_{\lambda_n\in \M}a_n w_{\alpha}(\lambda_n)r(\lambda_n)\right)I_2.
\end{align*}

\end{proof}

\begin{rmk}
 In the above analysis of $I_1$ and $I_2$, we may realize that $I_2$ is a positive real number and $I_1$ is `almost' a positive real number except for a small error. So the bound in Lemma~[\ref{lem:I1byI2}]  holds without the absolute value.
\end{rmk}

The parts of $I_1$ contributing to small values are estimated in the error term in the following lemma.
\begin{lem}\label{lem:simplifyI1}
We have
\begin{align*}
 I_1&=\left(\int_{Y_1}^{Y_2(\log Y_2)^2}+\int_{-Y_2(\log Y_2)^2}^{-Y_1}\right)\int_ {-X^{A_4}}^{X^{A_4}}F_0(x+u)\left(\frac{\sin \alpha u}{u}\right)^2e^{i2\alpha u}|R(x)|^2\Phi\left(\frac{x}{Y_2}\right)\d u\d x \\
 &+ O\left( \alpha  X^{A_3} e^{2M/C_1}\left(\sum_{\lambda_n\leq 4\alpha}a_n\right) + I_2X^{-A_4}\left(\sum_{n\leq X^{A_1}}a_n\right)\right).
\end{align*} 
\end{lem}
\begin{proof}
We write
\begin{align}
 I_1&=\left(\int_{Y_1}^{Y_2(\log Y_2)^2}+\int_{-Y_2(\log Y_2)^2}^{-Y_1}\right)\int_ {-X^{A_4}}^{X^{A_4}}F_0(x+u)\left(\frac{\sin \alpha u}{u}\right)^2e^{i2\alpha u}|R(x)|^2\Phi\left(\frac{x}{Y_2}\right)\d u\d x\\
 \notag
 &+ E_1 +E_2,
\end{align}
where 
\begin{align*}
 &E_1\\
 &=\left(\int^{\infty}_{Y_2(\log Y_2)^2}+\int_{-\infty}^{-Y_2(\log Y_2)^2}+\int^{Y_1}_{-Y_1}\right)\int_ {-\infty}^{\infty}F_0(x+u)\left(\frac{\sin \alpha u}{u}\right)^2e^{i2\alpha u}|R(x)|^2\Phi(x/Y_2)\d u\d x 
\end{align*}
 and
 \begin{align*}
 &E_2\\
 &=\left(\int^{\infty}_{Y_2(\log Y_2)^2}+\int_{-\infty}^{-Y_2(\log Y_2)^2}+\int^{Y_1}_{-Y_1}\right)\left(\int_ {X^{A_4}}^{\infty} + \int^{-X^{A_4}}_{-\infty}\right)\\
 &\hspace{5cm}F_0(x+u)\left(\frac{\sin \alpha u}{u}\right)^2e^{i2\alpha u}|R(x)|^2\Phi(x/Y_2)\d u\d x.
\end{align*}

As we have obtained (\ref{eq:tail_conv_formula}) in proof of Lemma~\ref{lem:I1byI2}, we may obtain
\begin{align*}
 E_2=O\left(I_2X^{-A_4}\sum_{n\leq X^{A_1}}a_n\right).
\end{align*}
First we use Lemma~\ref{lem:convolution_formula} and bound the covolution integral of $F_0(x+u)$ and then bound $E_1$ as below,
\begin{align*}
&E_1\\
&=\left(\int^{\infty}_{Y_2(\log Y_2)^2}+\int_{-\infty}^{-Y_2(\log Y_2)^2}+\int^{Y_1}_{-Y_1}\right)\int_ {-\infty}^{\infty}F_0(x+u)\left(\frac{\sin \alpha u}{u}\right)^2e^{i2\alpha u}|R(x)|^2\Phi(x/Y_2)\d u\d x \\
&\ll \left(\alpha\sum_{\lambda_n\leq 4\alpha}a_n\right)\left(\int_{Y_2(\log Y_2)^2}^{\infty}+\int^{-Y_2(\log Y_2)^2}_{-\infty} +\int^{-Y_1}_{Y_1}\right) |R(x)|^2\Phi(x/Y_2)\d x \\
& \ll \alpha  X^{A_3} \exp\left(\frac{2M}{C_1}\right)\sum_{\lambda_n\leq 4\alpha}a_n. 
\end{align*}
The final step comes from the observation that the contribution from the integral $\int^{-Y_1}_{Y_1}$ is the largest, and the bounds $Y_1\ll X^{A_3}$ and $|R(x)|^2\ll \exp\left(\frac{2M}{C_1}\right)$ from Lemma~\ref{lem:simplify0}.
\end{proof}

Lemma~\ref{lem:convolution_formula} relates $F_\beta(x)$ to $F_0(x)$. We further quantify this idea in the following lemma.

\begin{lem} \label{lem:Fbeta_to_F0}

We may relate $I_1$ with integral involving $F_\beta$ as follows:


\begin{align*}
&\left|\int_{Y_1}^{Y_2(\log Y_2)^2}\int_ {-X^{A_4}}^{X^{A_4}}F_\beta(x+u)\left(\frac{\sin \alpha u}{u}\right)^2e^{-i2\alpha u}|R(x)|^2\Phi(x/Y_2)\d u\d x\right|\\
 &=\frac{1}{2} I_1 + O\left(\alpha  X^{A_3} e^{2M/C_1}\left(\sum_{\lambda_n\leq 4\alpha}a_n\right) + I_2X^{-A_4}\left(\sum_{n\leq X^{A_1}}a_n\right)\right).
\end{align*}
\end{lem}

\begin{proof}
 
We use the bound (\ref{eq:tail_conv_formula}) of Lemma~\ref{lem:I1byI2} again and obtain
 
 \begin{align*}
&\left|\int_{Y_1}^{Y_2(\log Y_2)^2}\int_ {-X^{A_4}}^{X^{A_4}}F_\beta(x+u)\left(\frac{\sin \alpha u}{u}\right)^2e^{-i2\alpha u}|R(x)|^2\Phi(x/Y_2)\d u\d x\right|\\
&= \left|\int_{Y_1}^{Y_2(\log Y_2)^2}\int_ {-\infty}^{\infty}F_\beta(x+u)\left(\frac{\sin \alpha u}{u}\right)^2e^{-i2\alpha u}|R(x)|^2\Phi(x/Y_2)\d u\d x\right|\\
&+O\left(X^{-A_4}I_2\sum_{n\leq X^{A_1}} a_n\right).
\end{align*}
Now we have a chance to use Lemma~\ref{lem:convolution_formula} to absorb $e^{i\beta}$ in the absolute value and convert
$F_\beta$ to $F_0$. We obtain from the above
\begin{align*}
&=\frac{1}{2}\left|\int_{Y_1}^{Y_2(\log Y_2)^2}\int_ {-\infty}^{\infty}F_0(x+u)\left(\frac{\sin \alpha u}{u}\right)^2e^{-i2\alpha u}|R(x)|^2\Phi(x/Y_2)\d u\d x\right|\\
&+ \frac{1}{2}\left|\int^{-Y_1}_{-Y_2(\log Y_2)^2}\int_ {-\infty}^{\infty}F_0(x+u)\left(\frac{\sin \alpha u}{u}\right)^2e^{-i2\alpha u}|R(x)|^2\Phi(x/Y_2)\d u\d x\right|\\
&+O\left(X^{-A_4}I_2\sum_{n\leq X^{A_1}} a_n\right) \\
&\geq\frac{1}{2}\left|\int_{Y_1}^{Y_2(\log Y_2)^2}+\int^{-Y_1}_{-Y_2(\log Y_2)^2}\int_ {-\infty}^{\infty}F_0(x+u)\left(\frac{\sin \alpha u}{u}\right)^2e^{-i2\alpha u}|R(x)|^2\Phi(x/Y_2)\d u\d x\right|\\
&+O\left(X^{-A_4}I_2\sum_{n\leq X^{A_1}} a_n\right).
\end{align*}
We use Lemma~\ref{lem:simplifyI1} to complete the above expression to $I_1$ and obtain
\begin{align*}
& = \frac{1}{2}\left|\int_{-\infty}^{\infty}\int_ {-\infty}^{\infty}F_0(x+u)\left(\frac{\sin \alpha u}{u}\right)^2e^{-i2\alpha u}|R(x)|^2\Phi(x/Y_2)\d u\d x\right|\\
&+O\left(\alpha  X^{A_3} e^{2M/C_1}\left(\sum_{\lambda_n\leq 4\alpha}a_n\right) + I_2X^{-A_4}\left(\sum_{n\leq X^{A_1}}a_n\right)\right) \\
&= \frac{1}{2}I_1 + O\left(\alpha  X^{A_3} e^{2M/C_1}\left(\sum_{\lambda_n\leq 4\alpha}a_n\right) + I_2X^{-A_4}\left(\sum_{n\leq X^{A_1}}a_n\right)\right)
\end{align*}
 \end{proof}

\begin{proof}[\textbf{Proof of Theorem~\ref{thm:main}}]~

We may note, by the resonance method
\begin{align*}
 &\max_{X^{A_3}/2<x\leq 2A_2^2X^{A_2}(\log X)^2}|F_\beta(x)|\\
 &\geq \frac{1}{I_2\alpha}\left|\int_{Y_1/2}^{2Y_2(\log Y_2)^2}\int_ {-X^{A_4}}^{X^{A_4}}F_\beta(x+u)\left(\frac{\sin \alpha u}{u}\right)^2|R(x)|^2\Phi(x/Y_2)\d u\d x \right|\\
 &\geq \frac{I_1}{2I_2\alpha}+O\left( \frac{X^{A_3}e^{2M/C_1}}{I_2}\left(\sum_{\lambda_n\leq 4\alpha}a_n\right) + \frac{X^{-A_4}}{\alpha}\left(\sum_{n\leq X^{A_1}}a_n\right)\right),
 \end{align*}
 where the last lower bound is obtained using Lemma~\ref{lem:simplifyI1}. Now we use Lemma~\ref{lem:I1byI2} for a lower bound of $I_1/I_2$ and Lemma~\ref{lem:simplify0} for a lower bound of $I_2$ to obtain
 \begin{align*}
 &\geq \frac{1}{4\alpha}\sum_{\lambda_n\in\M}a_n r(\lambda_n)w_{\alpha}(\lambda_n) \\
&+O\left( X^{A_3-A_2}e^{\frac{14-C_1}{7C_1}M}\left(\sum_{\lambda_n\leq 4\alpha}a_n\right) + \frac{X^{-A_4}}{\alpha}\left(\sum_{n\leq X^{A_1}}a_n\right)\right)\\
&\geq \frac{\pi}{4e}\sum_{\lambda_n\in \M}a_n+ O\left(X^{A_3-A_2}e^{2M/C_1}\left(\sum_{\lambda_n\leq 4\alpha}a_n\right) + \frac{X^{-A_4}}{\alpha}\left(\sum_{n\leq X^{A_1}}a_n\right)\right),
\end{align*}
where we have used $r(\lambda_n)\geq e^{-1}$ and $w_{\alpha}(\lambda_n)\geq \alpha \pi$.
\end{proof}

\section{Applications}\label{sec:applications}
Below we discuss some applications of Theorem~\ref{thm:main}.

\subsection{Divisor Problem}(Proof of Theorem~\ref{thm:divisor})~

By Voronoi formula \cite{tit}, for large enough $X$ such that $\sqrt X \leq x \leq X^{3/2}$, we have
\begin{equation}
\label{eq:voronoi}
 \frac{\pi \sqrt 2}{\sqrt x}\Delta(x^2)=\sum_{n\leq X^3} \frac{d(n)}{n^{3/4}}\cos\left(4\pi \sqrt n x -\pi/4\right) + O(X^{\epsilon-\half}),
\end{equation}
for any $\epsilon>0$.

Let 
\begin{align*}
F_\beta(x)&=\sum_{n\leq X^3} \frac{d(n)}{n^{3/4}}\cos\left(4\pi \sqrt n x -\pi/4\right)\\
&=: \sum_{n\leq X^{A_1}} \frac{d(n)}{n^{3/4}}\cos\left(\lambda_n x + \beta\right),
\end{align*}

where $\lambda_n=4\pi \sqrt n, \, a_n=d(n)/n^{3/4}, \beta=-\pi/4, A_1=3$.

Further, let
\[A_2=3/2, A_3=1 \text{ and } A_4=7/8.\]

We construct the resonating set $\M$ following the construction of Soundarajan\cite{So1},
\[\M:=\left\{\lambda_n:n\in[C_1\alpha, 2 \alpha], \omega(n)=[\lambda\log_2\alpha], n \text{ square free}\right\},\]
where $\omega(n)$ denotes the number of distinct prime factors of $n$.

Note that
\begin{align*}
  \sum_{n\leq X^3} \frac{d(n)}{n^{3/4}} \ll X^{3/4}\log X.
 \end{align*}

 By Sathe's \cite{sathe} and Stirling's formula
 \[M\asymp \frac{\alpha}{\sqrt{\log_2 \alpha}}(\log \alpha)^{\lambda-1-\lambda\log \lambda}.\]
 Now 
 \begin{align*}
  \max_{X^{A_3}/2<x\leq 2A_2^2X^{A_2}(\log X)^2}|F_\beta(x)|
 &\geq \frac{\pi}{4e}\sum_{\lambda_n\in \M}\frac{d(n)}{n^{3/4}} + \text{Error} \\
 & \geq \frac{\pi}{2^{11/4}e}\frac{M2^{\lambda \log_2 \alpha}}{\alpha^{3/4}} + \text{Error}\\
 &\gg \frac{\alpha^{1/4}}{(\log_2 \alpha)^{1/2}}(\log \alpha)^{\lambda\log2+\lambda-1-\lambda\log\lambda} + \text{Error}.
\end{align*}
We choose 
$\alpha=\frac{1}{C}(\log X)(\log_2 X)^{1-\lambda+\lambda\log\lambda}(\log_3 X)^{1/2}$, obtaining
\begin{align*}
  \max_{X/2<x\leq 5X^{3/2}(\log X)^2}|F_\beta(x)|
  &\gg (\log X)^{1/4} (\log_2X)^{\lambda\log 2+ (3/4)(\lambda-1-\lambda\log\lambda)}(\log_3 X)^{-3/8}+ \text{Error}.
\end{align*}

On optimizing $\lambda=2^{4/3}$, we get
\[\max_{X/2<x\leq 5X^{3/2}(\log X)^2}|F_\beta(x)|
  \gg (\log X)^{1/4} (\log_2X)^{ (3/4)(2^{4/3}-1)}(\log_3 X)^{-3/8}+ \text{Error}.\]
  
To estimate the error, note
\begin{align*}
M&\asymp \frac{\alpha}{\sqrt{\log_2 \alpha}}(\log \alpha)^{\lambda-1-\lambda\log \lambda}\\
& \asymp \frac{\log X}{C}.
\end{align*}
We choose $C$ large enough such that  
\begin{align*}
 e^{2M/C}\ll X^{1/32-\epsilon} 
\end{align*}
for any $\epsilon>0$.

Now we compute the error as follows:
\begin{align*}
 \text{Error}&\ll X^{A_3-A_2}e^{2M/C_1}\left(\sum_{\lambda_n\leq 4\alpha}\frac{d(n)}{n^{3/4}}\right) + \frac{X^{-A_4}}{\alpha}\left(\sum_{n\leq X^{A_1}}
 \frac{d(n)}{n^{3/4}}\right)\\
 &\ll X^{-1/2+1/32-\epsilon}\alpha^{1/4}\log \alpha+\frac{X^{-7/8+3/4}\log X}{\alpha}\\
 &\ll X^{-15/16}+ X^{-1/8}\ll X^{-1/8}.
\end{align*}

This completes proof of Theorem~\ref{thm:divisor}.

\subsection{Circle Problem}(Proof of Theorem~\ref{thm:circle})~

For $\sqrt X\leq x\leq X^{3/2}$,
\[-\frac{\pi}{\sqrt x}P(x)=\sum_{n\leq X^3}\frac{r(n)}{n^{3/4}}\cos(2\pi\sqrt n x + \pi/4) + O(X^\epsilon).\]
So
\begin{align*}
F_\beta(x)&=: \sum_{n\leq X^{A_1}} \frac{r(n)}{n^{3/4}}\cos\left(\lambda_n x + \beta\right),
\end{align*}
where $\lambda_n=2\pi \sqrt n, \, a_n=r(n)/n^{3/4}, \beta=\pi/4, A_1=3$. We also choose $A_2=3/2, A_3=1 \text{ and } A_4=7/8$. 
\[\M=\left\{\lambda_n:n\in[C_1\alpha, 2 \alpha], n \text{ has }[\lambda\log_2\alpha]\text{ prime factors } p\equiv1\mod 4, n \text{ square free}\right\}.\]
By \cite{ten} and using Stirling's formula
\begin{align*}
M\asymp\frac{\alpha}{\log\alpha}\frac{(\half \log_2\alpha)^{[\lambda\log_2\alpha]-1}}{([\lambda \log_2\alpha]-1)!}\asymp \frac{\alpha}{\sqrt{\log_2 \alpha}}(\log \alpha)^{\lambda-1-\lambda\log \lambda-\lambda\log 2}.\\
\end{align*}
 Also note
\[r(m)\geq 2^{[\lambda\log_2 \alpha]}\gg(\log \alpha)^{\lambda\log 2}.\]
We may do the computation as before and obtain
\[\max_{X/2<x\leq 5X^{3/2}(\log X)^2}|F_\beta(x)|\gg \frac{\alpha^{1/4}}{(\log_2 \alpha)^{1/2}}(\log \alpha)^{\lambda-1-\lambda\log\lambda} + \text{Error}.\]
We choose 
$\alpha=\frac{1}{C}(\log X)(\log_2 X)^{1-\lambda+\lambda\log\lambda+\lambda\log2}(\log_3 X)^{1/2},$
and optimize $\lambda=2^{1/3}$, to get

\[\max_{X/2<x\leq 2X^{3/2}(\log X)^2}|F_\beta(x)|\gg \frac{(\log X)^{1/4}}{(\log_3 X)^{3/8}}(\log_2 X)^{\frac{3}{4}(2^{1/3}-1)} + \text{Error}.\]

To estimate the error, we notice that the choice of $\alpha$ is such that 
\[M\asymp \frac{\log X}{C},\]
and we choose $C$ large enough such that 
$ e^{2M/C}\ll X^{1/32-\epsilon}$. Again the calculation for the error term is similar to that of $\Delta(x)$, and we get
 \begin{align*}
 \text{Error}&\ll X^{A_3-A_2}e^{2M/C_1}\left(\sum_{\lambda_n\leq 4\alpha}\frac{r(n)}{n^{3/4}}\right) + \frac{X^{-A_4}}{\alpha}\left(\sum_{n\leq X^{A_1}}
 \frac{r(n)}{n^{3/4}}\right)\\
 &\ll X^{-1/2+1/32-\epsilon}\alpha^{1/4}\log \alpha+\frac{X^{-7/8+3/4}}{\alpha}\ll X^{-1/8}.
\end{align*}

\subsection{k-Divisor Problem}(Proof of Theorem~\ref{thm:piltz})~

By Soundararajan \cite{So}
 \begin{align*}
  \frac{\alpha^{1/k}}{\sqrt \pi}\int_{-\infty}^{\infty}&\Delta_k(x^ke^{u/x})e^{-u^2\alpha^{2/k}}du = O(x^{k/2-{3/5}}\alpha^{\half+\epsilon})\\
  &+\frac{x^{\frac{k-1}{2}}}{\pi \sqrt k}\sum_{n=1}^{\infty}\frac{d_k(n)}{n^{\frac{k+1}{2k}}}\exp(-\pi^2(n/\alpha)^{2/k})\cos\left(2\pi k n^{1/k}x+\frac{k-3}{4}\pi\right)
 \end{align*}

 For our application, we choose the range of $X$ as  $X/2\leq x\leq X^2$ and 
\[\alpha=\frac{1}{C(k)}\log X(\log_2 X)^{1+\lambda\log\lambda-\lambda}(\log_3 X)^{1/2}.\]
Also note that for such a large $\alpha$, the integral involving the Gaussian kernel behaves almost like the Dirac delta and approximate $\Delta_k(x^k)$. We may write
\[ \text{max}_{-1\leq h\leq 1}\Delta(x^ke^h)\geq\frac{\alpha^{1/k}}{\sqrt \pi}\int_{-\infty}^{\infty}\Delta_k(x^ke^{u/x})e^{-u^2\alpha^{2/k}}du + O(e^{-x}).\]
Truncating the infinite series at $n= X^{8/5}$,
\begin{align*}
 \text{max}_{-1\leq h\leq 1}\Delta(x^ke^h)\geq &x^{\frac{k-1}{2}}{\pi \sqrt k}\sum_{n=1}^{X^{8/5}}\frac{d_k(n)}{n^{\frac{k+1}{2k}}}\exp(-\pi^2(n/\alpha)^{2/k})\cos\left(2\pi k n^{1/k}x+\frac{k-3}{4}\pi\right)\\
 &+ O(x^{k/2-{3/5}+\epsilon}).
\end{align*}
As the above error is small, it is enough to consider  
\[F_\beta(x)=\sum_{n=1}^{X^2}\frac{d_k(n)}{n^{\frac{k+1}{2k}}}\exp(-\pi^2(n/\alpha)^{2/k})\cos\left(2\pi k n^{1/k}x+\frac{k-3}{4}\pi\right),\]

with $\lambda_n= 2\pi k n^{1/k},a_n=\frac{d_k(n)}{n^{\frac{k+1}{2k}}}\exp(-\pi^2(n/\alpha)^{2/k}), \beta=\frac{k-3}{4}\pi$ and $A_1=8/5$. Let $A_2=3/2, A_3=1 \text{ and } A_4=9/10$ and
let
\begin{align*}
\M:=\left\{\lambda_n:n\in[C_1(k)\alpha, 2 \alpha], \omega(n)=[\lambda\log_2\alpha], n \text{ square free}\right\},
\end{align*}
and as done in case of $\Delta(x)$,
\begin{align*}
M\asymp \frac{\alpha}{\sqrt{\log_2 \alpha}}(\log \alpha)^{\lambda-1-\lambda\log \lambda}.\\
\end{align*}
The choice of $\alpha$ is such that we make $M$ of order $\log X$. So we choose 
\[\alpha=\frac{1}{C(k)}(\log X)(\log_2 X)^{1+\lambda\log\lambda-\lambda}(\log_3 X)^\half.\]
Then 
\[M\asymp \frac{1}{C(k)}\log X.\]
Let $C(k)$ be large enough such that
\[e^{2M/C_1(k)}\ll X^{1/4}.\]
Note $d_k(m)\geq k^{[\lambda\log_2\alpha]}\gg(\log_2X)^{\lambda \log k}$. Putting it in Theorem~\ref{thm:main}, we obtain
\begin{align*}
\max_{X/2<x\leq 5X^{3/2}(\log X)^2}|F_\beta(x)|
&\gg\sum_{n\in \M}\frac{d_k(n)}{n^{\frac{k+1}{2k}}}\exp(-\pi^2(n/\alpha)^{2/k})+\text{Error}\\
&\gg \frac{\alpha^{1-\frac{k+1}{2k}}(\log_2X)^{\lambda\log k}}{(\log_2 \alpha)^{1/2}}(\log \alpha)^{\lambda-1-\lambda\log\lambda} + \text{Error}\\
&\gg \frac{(\log X)^{\frac{k-1}{2k}}}{(\log_3 X)^{\half-\frac{k-1}{4k}}}(\log_2 X)^{\lambda \log k+\frac{k+1}{2k}(\lambda-1-\lambda\log\lambda)} + \text{Error}.
\end{align*}

Optimizing $\lambda$ at $\lambda=k^{\frac{2k}{k+1}}$, we get

 \[\max_{X/2<x\leq 5X^{3/2}(\log X)^2}|F_\beta(x)|\gg \frac{(\log X)^{\frac{k-1}{2k}}}{(\log_3 X)^\frac{k+1}{4k}}(\log_2 X)^{\frac{k+1}{2k}(k^{\frac{2k}{k+1}}-1)} + \text{Error}.\]

Similar computation to $\Delta(x)$ shows
\begin{align*}
 \text{Error}
 &\ll X^{-1/2+1/4+\epsilon}+\frac{X^{-9/10+8(k-1)/(10k) +\epsilon}}{\alpha}\ll X^{-1/10}.
\end{align*}

We further notice, if $k\equiv 3\pmod 8$, $\beta=0$ and we obtain an $\Omega_+$ result; and if $k\equiv 7\pmod 8$, $\beta=\pi$ and we obtain an $\Omega_-$ result.

\subsection{Mean Square of the Riemann Zeta}(Proof of Theorem~\ref{thm:riemann})~

Consider $E_1(x) :=\frac{1}{\sqrt{2x}}E(2\pi x^2)$ for $x\geq 1$. It is sufficient obtain a lower bound for $E_1(x)$. Define
\begin{align*}
 P(x, \tau)&=\sum_{n\leq \tau^2}(-1)^nd(n)n^{-3/4}\cos(4\pi \sqrt n x)\left(1-|\frac{2\sqrt n}{\tau}|-1\right),\text{ and}\\
Q(x, \tau)&=\sum_{n\leq \tau^2}d(n)n^{-3/4}\cos(4\pi \sqrt n x)\left(1-|\frac{2\sqrt n}{\tau}|-1\right).
\end{align*}
Lau and Tsang \cite{Tsang} showed that
\begin{equation}\label{eq:tsangE1}
 \sup_{|y-x|\leq 1}|E_1(y)|\geq \half P(x, \tau) + O(\log \tau).
\end{equation}
They have also shown 
\begin{equation}\label{eq:TsangPQ}
 Q(x, \tau)=\sum_{j=1}^J\sum_{i=1}^Ja^jb^{-i}P(\sqrt 2^{j-i}x, \sqrt 2^{j+i}\tau)+O(\sqrt \tau),
\end{equation}
where $a=2^{1/4}-2^{-1/4}, b=2^{1/4}+2^{-1/4},$ and $J=2\log_2 \tau$.
We choose 
\[\tau\asymp (\log X)^{1/2}(\log_2 X)^{1/2(1-\lambda+\lambda\log\lambda)}(\log_3 X)^{1/4}.\]
Exact calculation for $F_\beta(x)$ in \cite{Tsang}, with choice of $\alpha=\tau^2$, can be carried out for $Q(x, \tau)$
to obtain
\[\max_{X/2<x\leq 5X^{3/2}(\log X)^2}Q(x, \tau)
  \gg (\log X)^{1/4} (\log_2X)^{ (3/4)(2^{4/3}-1)}(\log_3 X)^{-3/8}.\]
  From equation (\ref{eq:euler}), we obtain
  \[\max_{1\leq i, j\leq J}P(\sqrt 2^{j-i}x, \sqrt 2^{j+i}\tau)\gg Q(x, \tau)+O(\sqrt \tau).\]
  So
\[\max_{X/2<x\leq 5X^{3/2}(\log X)^4}P(x, \tau)
  \gg (\log X)^{1/4} (\log_2X)^{ (3/4)(2^{4/3}-1)}(\log_3 X)^{-3/8}.\]
  This gives our required lowerbound for $E(x).$
  
\subsection{Mean Square of the Dirichlet $L$-Function}(Proof of Theorem~\ref{thm:dirichlet})
For $x\geq 1$, define
\[E_1(q, x):=(2x)^{-1/2}E(q, 2\pi x^2).\]
Lau and Tsang\cite{Tsang2} proved that
\[\max_{X/2<x\leq X^{3/2}}|E_1(q, x)|\geq c'\frac{\phi(q)}{q^{3/4}}\left(\sum_{l|\prod_{r=1}^{\omega(q)}}l^{-1/4}\right)^{-1}P(x, \tau) + O(1),\]
 for some positive constant $c'$, and $P(x, \tau)$ and $\tau$ as defined for the mean square of the Riemann zeta function. Our required result follows from the above inequality.

\end{document}